\documentclass[11pt]{ip-journal}

\usepackage{amsmath, graphicx, cite, enumerate, comment}
\usepackage{amssymb, amsthm, amscd}
\usepackage{latexsym}
\usepackage{amssymb}
\usepackage[T1]{fontenc}	
\usepackage[english]{babel}
\usepackage{amsmath,amssymb}
\usepackage{amsfonts}
\usepackage{amscd}
\usepackage{cite}
\usepackage{marvosym}
\usepackage{mathrsfs}
\usepackage{amsthm}
\usepackage{verbatim}

\usepackage{tikz} \usetikzlibrary{arrows,shapes,patterns} 
\usepackage{lipsum} 

\newtheorem{thm}{Theorem}
\newtheorem{lemm}{Lemma}[section]

\newtheorem{prop}[thm]{Proposition}
\newtheorem*{thmA}{Theorem A}
\newtheorem*{thmB}{Theorem B}

\newtheorem*{prop0}{Proposition A.0}

\newtheorem*{conj}{Conjecture}

\theoremstyle{remark}
\newtheorem{rmk}{Remark}[section]

\theoremstyle{definition}

\newcommand{\be}{\begin{equation}}
\newcommand{\ee}{\end{equation}}
\newcommand{\bd}{\begin{displaymath}}
\newcommand{\ed}{\end{displaymath}}

     \title[Index and Betti numbers of minimal hypersurfaces]{Comparing the Morse index and the first Betti number of minimal hypersurfaces}
     \author{Lucas Ambrozio, Alessandro Carlotto and Ben Sharp}
     \address{Imperial College London, South Kensington Campus, London SW7 2AZ, United Kingdom}
     \email{l.ambrozio@imperial.ac.uk}
     \address{ETH Inst. f\"ur Theoretische Studien, Clausiusstrasse 47, 8092 Z\"urich, Switzerland}
     \email{alessandro.carlotto@eth-its.ethz.ch}
     \address{SNS Pisa, Piazza dei Cavalieri 7, 56126 Pisa, Italy}
     \email{benjamin.sharp@sns.it}
     
\begin{document}
     	
     	\begin{abstract}
     		By extending and generalising previous work by Ros and Savo, we describe a method to show that in certain positively curved ambient manifolds the Morse index of every closed minimal hypersurface is bounded from below by a linear function of its first Betti number. The technique is flexible enough to prove that such a relation between the index and the topology of minimal hypersurfaces holds, for example, on all compact rank one symmetric spaces, on products of the circle with spheres of arbitrary dimension and on suitably pinched submanifolds of the Euclidean spaces. These results confirm a general conjecture due to Schoen and Marques-Neves for a wide class of ambient spaces.
     	\end{abstract}
     	
     	\maketitle       
							
	\section{Introduction} 
	
		\indent Since Lawson \cite{Law} constructed closed embedded minimal surfaces of arbitrary genus in the round three-dimensional sphere, many ingenious ways of producing an abundance of closed embedded minimal hypersurfaces in general Riemannian manifolds have been discovered. \\
	\indent A few remarkable examples show how powerful and diverse these methods can be. While Hsiang and Lawson \cite{HsiLaw} have found all homogeneous minimal hypersurfaces of the round spheres in all dimensions, Hsiang \cite{HsiB, HsiC} (and Hsiang-Sterling \cite{HS}) discovered infinitely many embedded non-totally geodesic hyperspheres in round spheres of certain dimensions (in particular, in four dimensions). Kapouleas showed how to construct closed embedded minimal surfaces in three-manifolds with a generic metric by gluing and desingularization methods (see \cite{KapS} for a survey of these methods). Based on a min-max theory for the area functional developed by Almgren, Pitts \cite{Pit} and Schoen and Simon \cite{SchSim} proved that every $(n+1)$-dimensional closed Riemannian manifold contains at least one embedded minimal hypersurface (containing possibly a singular set of codimension seven). The basic idea is to apply the min-max variational approach to one-parameter families of sweep-outs of the ambient manifold by closed hypersurfaces. More recently, by using $k$-parameters sweep-outs, Marques and Neves \cite{MarNevB} proved that under the assumption of positiveness of the Ricci curvature, these ambient spaces contain in fact infinitely many embedded minimal hypersurfaces (which are, as in the previous result, smooth in low dimensions). A broad overview of these methods and of many of their intriguing applications are provided in the ICM lectures by Marques \cite{Mar} and, with more emphasis on mix-max techniques, by Neves \cite{Nev}.  \\
	\indent A different task is to understand the geometric and topological properties of this variety of examples. For instance, the construction by Kapouleas-Yang \cite{KapY} yields sequences of minimal surfaces with uniformly bounded area and unbounded genus. Simon and Smith \cite{Smi} (see also Colding and De Lellis \cite{ColDeL}) have proven that their min-max method applied to sweep-outs of the three dimensional sphere, endowed with an arbitrary Riemannian metric, by one-parameter families of two dimensional spheres produces an embedded minimal two-sphere. Still in three dimensions, Ketover \cite{Ket} proved that the genus of the min-max surface is controlled in an effective way by the genus of the surfaces in the sweep-outs (see also the previous work by De Lellis and Pellandini \cite{DeLPel}). \\
	\indent More generally, one could be interested in understanding the general properties of the set of embedded minimal hypersurfaces in a given manifold. For example, it has been shown by Choi and Schoen \cite{ChoSch} that in an ambient three-manifold with positive Ricci curvature, the set of all embedded minimal surfaces with genus bounded by a fixed constant is compact in the strongest sense.  \\
	\indent In addition to the most basic topological and geometric properties of closed minimal hypersurfaces, there is an important analytic quantity, the (Morse) index. Roughly speaking, the index of a minimal hypersurface counts the maximal number of directions the hypersurface can be deformed in such way that its volume is decreased. In many situations, the index of a minimal hypersurface controls its topology and geometry. The following examples are good illustrations of this phenomenon. Schoen and Yau \cite{SchYau} proved that an embedded closed orientable stable (= index zero) minimal surface in an orientable three-manifold with positive scalar curvature is necessarily a two-sphere. In these ambient manifolds, Chodosh, Ketover and M\'aximo \cite{CKM} have recently proved that the set of closed embedded minimal surfaces with bounded index cannot contain sequences of surfaces with unbounded area or genus (see also the work of the second-named author for the case of bumpy metrics \cite{CarB}). Going beyond three-dimensions, the compactness result of the third-named author \cite{Sha} (and later extended in the joint paper \cite{ACS}) is key in proving that, when the ambient $(n+1)$-manifold has positive Ricci curvature and $2\leq n \leq 6$, the set of closed embedded minimal hypersurfaces with a fixed bound on their index and volume contains only finitely many diffeomorphism types (more general results have been recently obtained by Buzano and the third-named author \cite{BS}, see also \cite{CKM}). There are even cases where closed minimal hypersurfaces can be classified by their index: Urbano \cite{Urb} proved that the only embedded minimal surfaces of the three-sphere that have index at most five are the totally geodesic equators and the Clifford tori. This contribution turned out to be crucial in the proof of the Willmore conjecture by Marques and Neves \cite{MarNevA}. In fact, it is expected that the index of a hypersurface obtained by min-max methods is bounded from above by the number of parameters of the family of sweep-outs considered and, under extra assumptions, is equal to such number. A significant contribution to this conjecture has been very recently presented in \cite{MarNevC}. As the above examples show, this type of results provide a very important tool to understand the min-max hypersurfaces, especially in higher dimensions.\\
	\indent In this article, we shall be concerned with the general problem of comparing two different ``measures of complexity'' of a minimal hypersurface inside a positively curved Riemannian manifold. In this respect, it has been conjectured that the index controls the basic topological invariants of minimal hypersurfaces in an effective way.
	\begin{conj}\label{indconj}[Schoen, Marques-Neves\cite{Mar, Nev}] Let $(\mathcal{N}^{n+1},g)$ be a closed Riemannian manifold with positive Ricci curvature. There exists a positive constant $C$ such that, for every closed embedded minimal hypersurface $M^n$, the following inequality holds
		\begin{equation*}
		index(M) \geq C b_{1}(M).
		\end{equation*}
	\end{conj}
	
	\indent In fact, we expect a similar conclusion to hold, under the very same assumptions on the ambient manifold, not only for $b_1(M)$ but for all Betti numbers of $M^n$, namely we expect the Morse index of $M^n$ to be bounded from below by a positive constant times the sum of the Betti numbers of the hypersurface in question. However, we remark that this assertion is actually \textsl{equivalent} to the aforementioned conjecture unless the dimension of the ambient manifold is greater or equal than five.

	\indent Notice that some positiveness assumption on the curvature is essential, as the manifolds $(\Sigma^2_\gamma\times S^{1},g+d\theta^2)$, where $(\Sigma^2_\gamma,g)$ is an orientable surface of arbitrary genus $\gamma \geq 2$ with a metric of constant Gaussian curvature $-1$, have non-positive sectional curvature and contain stable minimal surfaces of genus $\gamma$. \\
	\indent A particular example of an ambient manifold for which this conjecture has been already verified is the round sphere $S^{n+1}$ (of arbitrary dimension). The result is due to Savo (see \cite{SavA} and Theorem \ref{thmsavo} below). Roughly speaking, Savo's approach to the problem was to use each harmonic one-form on a minimal hypersurface $M^n$ to produce a set of functions that generate variations that decrease the area of the hypersurface (in mean). Then it was argued that if the index of $M^n$ were too small compared to its first Betti number (= dimension of the space of harmonic one-forms, by Hodge's Theorem), one  would reach a contradiction. \\
	\indent Savo's work was preceded by the work of Ros (see \cite{Ros} and references therein for partially similar contributions). Considering flat three dimensional tori, Ros observed that, when a harmonic one-form on a closed minimal surface is viewed as a vector field of $\mathbb{R}^3$ along the surface, its coordinates have the behaviour described above. In particular, he proved that the index of such surfaces is bounded from below by an affine function (whose coefficients do not depend on the particular surface) of their first Betti numbers (see Theorem 16 in \cite{Ros}). In \cite{Urb2}, Urbano used essentially the same method to prove the analogous result for closed minimal surfaces inside the product of the unit circle with a two-dimensional sphere of radius greater or equal than one. \\
	\indent The aim of the present paper is to extend this approach to other positively curved ambient manifolds, and to give a unified method to address the above conjecture within this framework. We work in the most general setting: given an ambient manifold $(\mathcal{N}^{n+1},g)$ that lies isometrically inside some Euclidean space $\mathbb{R}^d$ (for some $d$ big enough) and a minimal embedded hypersurface $M^n$ inside of it, we consider two kinds of test functions constructed from a harmonic one-form on $M^n$, both inspired by the works of Ros and Savo, both dependent on the specific isometric embedding of the ambient manifold in the Euclidean space (see Proposition \ref{propomega} and Proposition \ref{propomegaN} in Section 3). In particular, we obtain the following result.
	\begin{thmA}\label{thma}
		Let $(\mathcal{N}^{n+1},g)$ be a Riemannian manifold that is isometrically embedded in some Euclidean space $\mathbb{R}^d$. Let $M^{n}$ be a closed embedded minimal hypersurface of $(\mathcal{N}^{n+1},g)$. \\
		\indent Assume that for every non-zero vector field $X$ on $M^n$,
		\begin{multline} \label{desthm0}
		\int_{M} \left[tr_M (Rm^{\mathcal{N}}(\cdot, X,\cdot,X)) + Ric^{\mathcal{N}}(N,N)|X|^2\right] dM\\
		>\int_{M} \left[(|II(\cdot,X)|^2 -|II(X,N)|^2) + (|II(\cdot,N)|^2-|II(N,N)^2|)|X|^2\right]dM
		 \end{multline}
		\noindent where $Rm^{\mathcal{N}}$ denotes the Riemann curvature tensor of $\mathcal{N}^{n+1}$, $II$ denotes the second fundamental form of $\mathcal{N}^{n+1}$ in $\mathbb{R}^d$ and $N$ is a local unit normal vector field on $M^n$. \\
		\indent Then
		\begin{equation*}
		index(M) \geq \frac{2}{d(d-1)}b_{1}(M).
		\end{equation*} 
	\end{thmA}
	\indent See Sections 2 and 3 for more details, and also for the precise geometric meaning of the inequality \eqref{desthm0}. \\
	\indent Specializing the above result, we were able to check that the conjecture quoted above is true in a number of ambient spaces. In particular, we prove the conjecture for the projective spaces $\mathbb{RP}^{n+1}, \mathbb{
		CP}^m,\mathbb{HP}^p, Ca\mathbb{P}^2$ endowed with their standard metrics. Together with the round sphere, these spaces comprise all compact rank one symmetric spaces, i.e., all compact symmetric spaces with positive sectional curvature (see \cite{Pet}).  
	\begin{thmB}
		Any closed, embedded minimal hypersurface of a compact rank one symmetric space is such that its first Betti number is bounded from above by a constant times its index. The constant depends only on the dimension of the ambient manifold.
	\end{thmB}
	\indent As we shall see in Subsection \ref{subs:rankone} proving the index conjecture for $\mathbb{C}\mathbb{P}^n$ turns out to be rather subtle, as one needs to handle (in fact: to rule out) the borderline case when equality holds in \eqref{desthm0}. This requires a delicate \textsl{ad hoc} argument, which is presented in Appendix \ref{app:borderline}. \\
	\indent Going beyond highly symmetric examples and the positive Ricci curvature assumption, we also verified that the same control of the first Betti number by the index also holds for minimal surfaces in sufficiently pinched convex hypersurfaces of the Euclidean spaces (Theorem \ref{thmellip}), in three-manifolds that satisfy a pinching condition for the scalar curvature (Theorem \ref{thmpinch3}) and in the product of circles with spheres of any dimension (Theorem \ref{thmprod}). In fact, our methods also allow to prove the index estimate for products of spheres  $S^{p}\times S^{q}$ unless both factors have dimension equal to two (Theorem \ref{thmspheres}).\\
	\indent Furthermore, we would like to point out that when $n=2$ and $Ric^\mathcal{N} >0$, the area estimate of Choi-Wang \cite{ChoWan} and Choi-Schoen \cite{ChoSch} coupled with our upper bound on the first Betti number yields an effective affine area bound on any minimal surface in $(\mathcal{N}^3,g)$ solely in terms of the index. \\
	\indent  We conclude this introduction with a few remarks. Firstly, we observe that if the pinching condition in Proposition \ref{propmain} holds for certain ambient manifolds $(\mathcal{N}^{n+1},g)$ in $\mathbb{R}^d$, then it will also hold, possibly with a slightly larger $\eta$, for small $C^2$ perturbations of $(\mathcal{N}^{n+1},g)$ in $\mathbb{R}^{d+k}$, $k\geq 0$. This in particular shows that the method is flexible enough to deal with isometric embeddings that are definitely not ``rigid'' nor very symmetric, which seems to be a new and peculiar feature of our approach. Secondly, although the methods used in this paper indicate that, under certain curvature conditions of the ambient manifold, the topology of a minimal hypersurface contributes to its index, the converse is not true. For example, the infinitely many embedded minimal three-spheres in the round four-sphere discovered by Hsiang have uniformly bounded area but their indexes are not uniformly bounded (this fact can be seen as consequence of the the third-named author's compactness theorem \cite{Sha}, see \cite{CarA} for further discussion and generalizations). Finally, the case of the complex projective space (see Theorem \ref{thm:compl}) has special interest, since it shows that our results are somehow peculiar to the codimension one scenario as there are plenty of algebraic curves of any genus in the complex projective space that are area-minimising.

	\
	
	\noindent \textit{Acknowledgements:} The authors wish to express their gratitude to Andr\'e Neves for the inspiration for this work and a number of enlightening conversations, to Alessandro Savo for useful correspondence concerning some extensions of his results and to Fernando Cod\'a Marques for his interest in this work. \\
	\indent This project started whilst A.C and B.S. were supported by Professor Neves' ERC Start Grant agreement number P34897, The Leverhulme Trust and the EPSRC Programme Grant entitled  `Singularities of Geometric Partial Differential Equations', and L.A. was supported by CNPq - Conselho Nacional de Desenvolvimento Cient\'ifico e Tecnol\'ogico - Brazil. Since 1st October 2015, B.S. has held a junior visiting position at the Centro di Ricerca Matematica Ennio De Giorgi and would like to thank the centre for its support and hospitality. Since 1st November 2015, L.A. has held a Research Associate position at Imperial College under the Professor Neves' ERC Start Grant PSC and LMCF 278940. Lastly, this work was completed while A. C. was an ETH-ITS fellow: the outstanding support of Dr. Max R\"ossler, of the Walter Haefner Foundation and of the ETH Zurich Foundation are gratefully acknowledged.
	
	\section{Basic material}
	
	\indent In this section, we set some notations that will be used throughout the paper and recall some relevant definitions and results.
	
	\subsection{Two-sided and one-sided hypersurfaces} Let $(\mathcal{N}^{n+1},g)$ be a complete Riemannian manifold. Let $M^{n}$ be a closed, immersed hypersurface in $\mathcal{N}^{n+1}$. We can distinguish these immersions by their normal bundles: $M^n$ is called two-sided when this bundle is trivial, and one-sided otherwise. \\
	\indent When $M^n$ is two-sided, we can choose a smooth normal unit vector field $N$ along $M^n$. When $M^{n}$ is one-sided, we can construct a two-to-one cover $\pi: \hat{M}^{n} \rightarrow M^{n}$ that is a two-sided immersion of $\hat{M}^n$ into $\mathcal{N}^{n+1}$. More precisely, $\hat{M}^n$ is the set of pairs $(x,N)$, where $x$ belongs to $M^n$ and $N$ is some unit vector in $T_{x}\mathcal{N}$ normal to $T_{x}M$, $\pi$ is the obvious projection, the deck transformation $\tau$ sends $(x,N)$ to $(x,-N)$ and the field $\hat{N}((x,N))=N$ gives a trivialization of the normal bundle of the immersion $\pi: \hat{M}^{n} \rightarrow M^n \subset \mathcal{N}^{n+1}$. \\
	\indent When $\mathcal{N}^{n+1}$ is orientable, $M^{n}$ is one-sided if and only if $M^{n}$ is non-orientable, in which case the above construction defines precisely the oriented double cover of $M^n$. \\
	\indent In order to handle the case when $M^n$ may not be orientable, we convene to denote by $dM$ the Riemannian density of $M^n$ (for which we refer the reader, for instance, to the last section in chapter 14 of \cite{Lee}). We shall remind the reader that the divergence theorem does hold true in such setting (see e. g. Theorem 14.34 in \cite{Lee}), which enables us to perform integration by parts whenever needed. In any event, this is only relevant for Subsection \ref{subs:realproj}, \ref{subs:circle} and \ref{subs:threeman} as in all other examples we are about to analyze the manifold $\mathcal{N}^{n+1}$ is simply-connected, which ensures that any closed embedded hypersurface $M^n$ in $\mathcal{N}^{n+1}$ is automatically two-sided and orientable.\\

	\subsection{Isometric embeddings} By Nash's embedding theorem, we can consider any ambient manifold $(\mathcal{N}^{n+1},g)$ to be isometrically embedded in some Euclidean space $\mathbb{R}^{d}$ of sufficiently high dimension $d$. Let $D$ denote the Levi-Civita connection of the Euclidean space and $\nabla$ denote the Levi-Civita connection of $(\mathcal{N}^{n+1},g)$. The relationship between them is given by the formula
	\begin{equation*}
	D_{X}Y = \nabla_{X}Y + II(X,Y),
	\end{equation*}
	\noindent where $X,Y$ are vectors fields tangent to $\mathcal{N}^{n+1}$ and $II(X,Y)$ is a section of the normal bundle of $\mathcal{N}^{n+1}$ in $\mathbb{R}^{d}$. $II$ is the second fundamental form of $\mathcal{N}^{n+1}$ in $\mathbb{R}^d$. \\
	\indent A useful formula is the Gauss equation for the embedding of $(\mathcal{N}^{n+1},g)$ in $\mathbb{R}^d$: for all vector fields $X, Y$ on $\mathcal{N}^{n+1}$,
	\begin{equation} \label{eqGaussN}
	Rm^{\mathcal{N}}(X,Y,X,Y) = \langle II(X,X),II(Y,Y)\rangle - |II(X,Y)|^2,
	\end{equation}
	\noindent where $Rm^{\mathcal{N}}$ denotes the Riemann curvature tensor of $(\mathcal{N}^{n+1},g)$. According to our convention, $Rm^{\mathcal{N}}(X,Y,X,Y)$ gives the sectional curvature of the two-dimensional plane generated by $X$ and $Y$ if $X,Y$ are orthonormal. \\
	\indent Similarly, we have the Gauss equation for $M^n$ inside of $(\mathcal{N}^{n+1},g)$,
	\begin{equation} \label{eqGaussM}
	Rm^{M}(X,Y,X,Y) = Rm^{\mathcal{N}}(X,Y,X,Y) + \langle A(X,X),A(Y,Y)\rangle - |A(X,Y)|^2.
	\end{equation}
	\noindent where $A$ denotes the second fundamental form of $M^n$ in $\mathcal{N}^{n+1}$. According to our conventions, when $M^n$ is two-sided with unit normal field $N$, $A$ is given by $A(X,Y)= -g(\nabla_{X}N,Y)N$ for all $X,Y$ on $M^n$. The induced Riemannian connection on $M^n$ will be denoted by $\nabla^M$.\\
	
	\subsection{The Morse index} A closed embedded hypersurface $M^n$ in $(\mathcal{N}^{n+1},g)$ is called minimal when the first variation of its $n$-dimensional volume is zero for all variations generated by flows of vector fields $X\in \mathcal{X}(\mathcal{N})$. Equivalently, $M^n$ is minimal when the trace of its second fundamental form is identically zero. \\
	\indent The index form of an embedded minimal hypersurface $M^n$ in $\mathcal{N}^{n+1}$ is the quadratic form $Q$ on the set of smooth sections $W$ of the normal bundle of $M^n$ in $\mathcal{N}^{n+1}$ defined by
	\begin{align*}
	Q(W,W) & = - \int_{M} \langle\mathcal{J}_{M}(W),W\rangle dM 
	\\
	& = \int_{M} |\nabla_M^{\perp} W|^2 - (Ric^{\mathcal{N}}(W,W) + |A|^2|W|^2) dM,
	\end{align*}
	\noindent where $\mathcal{J}_{M}$ is the Jacobi operator acting on the normal bundle of $M^n$,
	\begin{equation*}
	\mathcal{J}_{M}W = \Delta_{M}^{\perp}W + Ric^{\mathcal{N}}(W)^{\perp} + |A|^2W.
	\end{equation*}
	\indent In the above formulae, $\nabla_M^{\perp}$ denotes the connection of the normal bundle of $M^n$, $\Delta_M^{\perp}$ is the Laplacian of this connection, $A$ denotes the second fundamental form of $M^n$ and the Ricci tensor of $\mathcal{N}^{n+1}$ is viewed as an endomorphism of the tangent bundle of $\mathcal{N}^{n+1}$. \\
	\indent In fact, $Q(W,W)$ gives precisely the second variation of the volume of $M^{n}$ at $t=0$ under a flow $\phi_{t}$ generated by any vector field $X\in\mathcal{X}(\mathcal{N})$ that coincides with $W$ on $M^n$ (see e. g. \cite{CM}). \\
	\indent The index of a closed embedded minimal hypersurface $M^n$ of $\mathcal{N}^{n+1}$ is the maximal dimension of a vector space of sections of its normal bundle restricted to which the index form is negative definite. Geometrically, the index measures how many directions one can deform $M^{n}$ to decrease its volume. \\
	\indent The index can be more easily computed when $M^{n}$ is two-sided. In this case, the sections of the normal bundle can be identified with the set of smooth functions $\phi$ on $M^n$, the index form of $M^n$ corresponds to the quadratic form
	\begin{equation} \label{defindexfrom}
	Q(\phi,\phi) = \int_{M} |\nabla^{M} \phi|^2 - (Ric^{\mathcal{N}}(N,N) + |A|^2)\phi^2 dM,
	\end{equation}
	\noindent and the Jacobi operator can be seen as the Schr\"odinger type operator
	\begin{equation} \label{defijacobi}
	J_M \phi = \Delta_{M}\phi + Ric(N,N)\phi + |A|^2\phi.
	\end{equation} 
	\noindent The index of $M^n$ is then the number of negative eigenvalues of $J_M$. \\
	\indent When $M^n$ is one-sided, the sections of its normal bundle can be identified with the \textit{odd} functions on the two-sided cover $\hat{M}^n$, i. e. the smooth functions $\phi$ on $\hat{M}^n$ such that $\phi\circ \tau = -\phi$. One can then compute the index of $M^n$ by counting the number of negative eigenvalues of $J_{\hat{M}}$ restricted to the space of odd functions on $\hat{M}^n$ (see \cite{Urb}).\\
	
	\subsection{Harmonic forms}\label{subs:harm} (see, for example, \cite{Pet}). In a closed Riemannian manifold $(M^n,g)$, the Ho\-dge-\-La\-pla\-ce operator is the second order differential operator $\Delta_{p}$ acting on $p$-forms defined by
	\begin{equation*}
	\Delta_p = dd^{*} + d^{*}d
	\end{equation*}  
	\noindent where $d : \Omega^{p}(M) \rightarrow \Omega^{p+1}(M)$ is the exterior differential and $d^{*}:  \Omega^{p}(M) \rightarrow \Omega^{p-1}(M) $ is the formal adjoint\footnote{It is well-known that the operator $d^{*}$ satisfies the identity $d^{\ast}=(-1)^{n(p+1)+1}\ast d \ast$ (where $\ast$ stands for the Hodge star operator) so that, as a result, both $d^{\ast}$ and $\Delta_p$ are globally well-defined even when $M$ is not orientable.} of $d$, defined with respect to the metric $g$. A $p$-form $\omega$ is called harmonic when $\Delta_{p}\omega = 0$.\\
	\indent As $M^n$ is closed, $\omega$ is harmonic if and only if it is closed and co-closed, i.e., if it satisfies the two equations $d\omega=0$ and $d^{*}\omega = 0$. Hodge's Theorem asserts that in a closed Riemannian manifold every De Rham cohomology class contains precisely one harmonic representative. Hence, the dimension of the space of harmonic $p$-forms coincides with the $p$-th Betti number of $M^n$. \\
	\indent The Bochner-Weitzenb\"ock formula relates the Hodge-Laplace operator with the usual (rough) Laplacian on forms:
	\begin{equation*} 
	\Delta_{p}\omega = - \Delta\omega + \mathcal{R}_{p}(\omega),
	\end{equation*}
	\noindent where $\mathcal{R}_{p}$ is a zero-th order curvature term. Such term, while rather complicated when $1<p<n$, has a remarkably simple expression when $p=1$, in which case this formula becomes
	\begin{equation} \label{eqbochner}
	\Delta_{1}\omega = - \Delta\omega + Ric^{M}(\omega^{\sharp},\cdot).
	\end{equation} 
	
	\begin{rmk}
		Here and in the following we use the usual musical isomorphisms to pass from vectors to one-forms. For example, if $\omega$ is a one-form on $(M^n,g)$, $\omega^{\sharp}$ is the unique vector field on $M^n$ such that $\omega(Y)=g(\omega^{\sharp},Y)$ for all vector fields $Y$. If $X$ is a vector, $\omega$ a one-form and $\theta$ a two-form, then $g(X^{\flat}\wedge\omega,\theta)=g(\omega,i_{X}\theta)$, where $i_{X}\theta$ denotes the contraction of the two form $\theta$ by the vector $X$, i.e., the one-form defined by $i_{X}\theta(Y)=\theta(X,Y)$ for all vector fields $Y$.
	\end{rmk}
	
	\indent When $(M^n,g)$ is isometrically immersed in $(\mathcal{N}^{n+1},g)$, it is possible to express the curvature term $\mathcal{R}_p$ of $M^n$ in terms of the corresponding operator in $(\mathcal{N}^{n+1},g)$ and the second fundamental form $A$ (for explicit formulae, see \cite{SavB}).  \\

	\section{The computations}
	
	\indent We describe two methods for using the harmonic one-forms of a minimal hypersurface to generate interesting test functions for the index form. The methods were inspired by the work of Ros \cite{Ros}, Urbano \cite{Urb} and Savo \cite{SavA}. \\
	\indent If $M^n$ is a two-sided minimal hypersurface of $\mathcal{N}^{n+1}\subset\mathbb{R}^d$, the coordinates in $\Lambda^{p}\mathbb{R}^{d}$ of any $p$-form on $M^n$ can be used to produce globally defined functions on $M^n$, which in turn can be used as test functions for the index form of $M^n$. We perform such computation in two cases: the coordinates of a harmonic one-form $\omega$ on $M^n$, and the coordinates of the two-form $N^{\flat}\wedge\omega$, for $\omega$ a harmonic one-form on $M^n$.
	
	\begin{rmk} Here and in all the following sections, $\{e_1,\ldots, e_n\}$ will denote an arbitrary local orthonormal frame on the hypersurface $M^n$. It is immediate to check that all quadratic expressions involving summations on such basis are independent of the particular choice of the basis itself, which implies that such quantities are globally defined on $M^n$. 
		\end{rmk}

	\subsection{First method: coordinates of $\omega$} 
	
	\begin{prop} \label{propomega}
		Let $(\mathcal{N}^{3},g)$ be a Riemannian three-manifold isometrically embedded in some Euclidean space $\mathbb{R}^{d}$. Let $M^2$ be a closed two-sided immersed minimal surface of $\mathcal{N}^{3}$. Given a harmonic one-form $\omega$ on $M^2$, let 
		\begin{equation*}
		u_{i} = \langle\omega^{\sharp},\theta_i\rangle, \quad i=1,\ldots, d,
		\end{equation*}
		\noindent denote the coordinates of $\omega^{\sharp}$ in $\mathbb{R}^{d}$ with respect to some orthonormal basis $\{\theta_i\}_{i=1}^{d}$ of $\mathbb{R}^{d}$. Then
			\begin{equation} \label{eqomega}
			\sum_{i=1}^{d}Q(u_i,u_i) = \int_{M} \left[\sum_{k=1}^{2} |II(e_k,\omega^{\sharp})|^2 - \frac{R^{\mathcal{N}}}{2}|\omega|^2\right] dM.
			\end{equation}
	\end{prop}
	\begin{proof}
		\indent Let $\{e_1,e_2\}$ be a local orthonormal frame on $M^2$. The functions $u_{i}=\langle\omega^{\sharp},\theta_{i}\rangle$ are such that
		\begin{equation*}
		D_{e_k}u_{i} = \langle D_{e_{k}}\omega^{\sharp},\theta_i\rangle.
		\end{equation*} 
		\indent Thus, plugging the functions $u_{i}=\langle\omega^{\sharp},\theta_{i}\rangle$ in the index form (\ref{defindexfrom}) and summing up on $i=1,\ldots,d$ gives
		\begin{equation*}
		\sum_{i=1}^{d}Q(u_i,u_i) = \int_{M} \sum_{i=1}^{2}|D_{e_i}\omega^{\sharp}|^2 -(|A|^2 + Ric^{\mathcal{N}}(N,N))|\omega|^2 dM.
		\end{equation*}
		\indent Since we have the orthogonal decomposition
		\begin{equation*}
		D_{e_k}\omega^{\sharp} = \nabla_{e_k}^{M}\omega^{\sharp} + A(e_k,\omega^{\sharp}) + II(e_i,\omega^{\sharp})
		\end{equation*}
		\noindent for each $k=1,2$, it follows that
		\begin{equation*} 
		\sum_{k=1}^{2}|D_{e_k}\omega^{\sharp}|^2 = \sum_{k=1}^{2}|\nabla^{M}_{e_k} \omega^{\sharp}|^2 + \sum_{k=1}^{2} |A(e_k,\omega^{\sharp})|^2 + \sum_{i=1}^{2}|II(e_k,\omega^{\sharp})|^2. 
		\end{equation*}
		\indent Thus, we obtain
		\begin{multline*}
		\sum_{i=1}^{d}Q(u_i,u_i) = \int_{M} |\nabla^{M} \omega|^2 + \sum_{i=1}^{2} |A(e_k,\omega^{\sharp})|^2 + \sum_{i=1}^{2}|II(e_k,\omega^{\sharp})|^2dM \\
		-\int_{M} (|A|^2 + Ric^{\mathcal{N}}(N,N))|\omega|^2 dM.
		\end{multline*}
		\indent Fixing a point on $M^2$ and choosing $\{e_1,e_2\}$ in such way that the second fundamental form of $M^2$ in $\mathcal{N}^3$ is diagonalised (i.e., $A(e_i,e_j)=k_i \delta_{ij}N$ for each $i,j=1,2$), it can be checked that
		\begin{equation*}
		\sum_{i=1}^{2}|A(e_{i},\omega^{\sharp})|^2=k_1^2\langle e_1,\omega^{\sharp}\rangle^2+k_2^2\langle e_2,\omega^{\sharp}\rangle^2=\frac{1}{2}|A|^{2}|\omega|^2,
		\end{equation*}
		\noindent since $M^2$ is minimal. Hence,
		\begin{multline} \label{eqauxpropomega}
		\sum_{i=1}^{d}Q(u_i,u_i) = \int_{M} \left(|\nabla^{M} \omega|^2 + \sum_{i=1}^{2}|II(e_k,\omega^{\sharp})|^2\right)dM \\
		-\int_{M} \left(\frac{1}{2}|A|^2 + Ric^{\mathcal{N}}(N,N)\right)|\omega|^2 dM.
		\end{multline}
		\indent Contraction of the Gauss formula (\ref{eqGaussM}) for the minimal surface $M^2$ gives the identity
		\begin{equation*}
		2K^{M} = R^{\mathcal{N}} - 2Ric^{\mathcal{N}}(N,N) - |A|^2,
		\end{equation*}
		\noindent where $K^M$ is the Gaussian curvature of $M^2$. Moreover, since $\omega$ is harmonic, integration of the Bochner-Weitzenb\"ock formula for $p=1$ (equation \eqref{eqbochner}) gives
		\begin{equation*}
		\int_{M} |\nabla^{M}\omega|^2 dM = - \int_{M} Ric^{M}(\omega^{\sharp},\omega^{\sharp}) dM. 
		\end{equation*}
		\indent Since $Ric^{M}(\omega^{\sharp},\omega^{\sharp})=K^{M}|\omega|^2$ as $M^2$ is two-dimensional, formula (\ref{eqomega}) follows now from substituting the above two identities into (\ref{eqauxpropomega}).
	\end{proof}

	\
	
	\subsection{Second method: coordinates of $N^{\flat}\wedge \omega$}
	
	\begin{prop} \label{propomegaN}
		Let $(\mathcal{N}^{n+1},g)$ be a Riemannian manifold isometrically embedded in some Euclidean space $\mathbb{R}^{d}$. Let $M^n$ be a closed two-sided immersed minimal hypersurface of $\mathcal{N}^{n+1}$. Given a harmonic one-form $\omega$ on $M^n$, let 
		\begin{equation*}
		u_{ij} = \langle N\wedge\omega^{\sharp},\theta_{ij}\rangle, \quad i,j = 1,\ldots, d, \, i < j, 
		\end{equation*}
		\noindent denote the coordinates of $N\wedge\omega^{\sharp}$ in $\Lambda^{2}\mathbb{R}^{d}$ with respect to some orthonormal basis $\{\theta_{ij}\}_{i<j}$ of $\Lambda^2\mathbb{R}^{d}$. Then
		\begin{multline} \label{eqomegaN}
		 \sum_{i<j}^{d}Q(u_{ij},u_{ij}) = \int_{M} \left[\sum_{k=1}^{n} |II(e_k,\omega^{\sharp})|^2 + \sum_{k=1}^{n} |II(e_k,N)|^2|\omega|^2\right]dM \\
		 - \int_{M} \left[\sum_{k=1}^{n} Rm^{\mathcal{N}}(e_k,\omega^{\sharp},e_k,\omega^{\sharp}) + Ric^{\mathcal{N}}(N,N)|\omega|^2\right] dM.
		\end{multline}
	\end{prop}
	\begin{proof}
		Let $\{e_1,\ldots,e_n\}$ be a local orthonormal frame on $M^n$. Observe that
		\begin{equation*}
		|\nabla^{M}u_{ij}|^2 = \sum_{k=1}^{n}|D_{e_k}u_{ij}|^2 = \sum_{k=1}^{d}\langle D_{e_k}(N\wedge\omega^{\sharp}),\theta_{ij}\rangle^2. \\
		\end{equation*}
		\indent Hence, summing up the values of the index form (\ref{defindexfrom}) on the functions $u_{ij}$ for all $i<j$, we obtain
		\begin{align} \label{eqauxpropomegaN}
		\sum_{i<j}^{d} Q(u_{ij},u_{ij}) & = \int_{M} \sum_{k=1}^{n}|D_{e_k}(N\wedge\omega^{\sharp})|^2-(Ric^{\mathcal{N}}(N,N)+|A|^2)|N\wedge\omega^{\sharp}|^2dM \\
		\nonumber
		& =\int_{M} \sum_{k=1}^{n}|D_{e_k}(N^{\flat}\wedge\omega)|^2-(Ric^{\mathcal{N}}(N,N)+|A|^2)|N^{\flat}\wedge\omega|^2dM.
	\end{align}
		\indent Obviously, $|N^{\flat}\wedge\omega|=|\omega|$ since $N$ is a unit vector field that is orthogonal to $\omega^{\sharp}$.  Moreover,
		\begin{align*}
		|D_{e_k}(N^{\flat}\wedge\omega)|^2  = &  |(D_{e_k}N)^{\flat}\wedge\omega + N^{\flat}\wedge D_{e_k}\omega|^2 \\
		 = & \langle(D_{e_k}N)^{\flat}\wedge\omega,(D_{e_k}N)^{\flat}\wedge\omega\rangle \\
		& + 2\langle D_{e_k}N^{\flat}\wedge\omega,N^{\flat}\wedge D_{e_k}\omega\rangle \\
		&   + \langle N^{\flat}\wedge D_{e_k}\omega,N^{\flat}\wedge D_{e_k}\omega\rangle \\
		= & \langle\omega,i_{D_{e_{k}N}}((D_{e_k}N)^{\flat}\wedge\omega)\rangle \\
		& + 2\langle i_{N}((D_{e_k}N)^{\flat}\wedge\omega),D_{e_k}\omega\rangle \\
		&   + \langle D_{e_k}\omega,i_N(N^{\flat}\wedge D_{e_k}\omega)\rangle \\
		= & \langle \omega,(|D_{e_k}N|^2\omega - (D_{e_k}N)^{\flat}\wedge i_{D_{e_{k}}N}\omega)\rangle \\
		& + 2\langle\langle D_{e_k}N,N\rangle\omega-(D_{e_k}N)^{\flat}\wedge i_{N}\omega,D_{e_k}\omega\rangle \\
		& + \langle D_{e_k}\omega,(D_{e_k}\omega - N^{\flat}\wedge i_{N}D_{e_k}\omega)\rangle \\
		 = & |D_{e_k}N|^2|\omega|^2 - |i_{D_{e_{k}}N}\omega|^2 + |D_{e_k}\omega|^2-|i_{N}D_{e_k}\omega|^2.
		\end{align*}
		\indent The pairing between vectors and one-forms allow us to rewrite the contractions in the above formula as
		\begin{align*}
		|D_{e_k}(N^{\flat}\wedge\omega)|^2 & = |D_{e_k}N|^2|\omega|^2- \langle \omega^{\sharp},D_{e_k}N\rangle^2 + |D_{e_k}\omega|^2-\langle D_{e_k}\omega^{\sharp},N\rangle^2  \\
		& = |D_{e_k}\omega|^2- 2\langle \nabla_{e_k}\omega^{\sharp},N\rangle^2 + |D_{e_k}N|^2|\omega|^2 \\
		& = |D_{e_k}\omega|^2- 2|A(e_{k},\omega^{\sharp})|^2 + |D_{e_k}N|^2|\omega|^2.
		\end{align*}
		\indent As one has the orthogonal decompositions
		\begin{equation*}
		D_{e_k}\omega^{\sharp} = \nabla^{M}_{e_k}\omega^{\sharp} + A(e_k,\omega^{\sharp})N + II(e_k,\omega^{\sharp}) 
		\end{equation*}
		and
		\begin{equation*}
		  D_{e_k}N=\nabla_{e_k} N + II(e_k,N),
		\end{equation*}
		\noindent it follows that
		\begin{multline*}
		|D_{e_k}(N^{\flat}\wedge\omega)|^2 \\  = |\nabla^{M}_{e_k}\omega|^2 - |A(e_{k},\omega^{\sharp})|^2 + |II(e_k,\omega^{\sharp})|^2 + |\nabla_{e_k}N|^2|\omega|^2 + |II(e_k,N)|^2|\omega|^2.
		\end{multline*}
		\indent Summing up, 
		\begin{align*}
		\sum_{k=1}^{n} |D_{e_k}(N^{\flat}\wedge\omega)|^2 & =   |\nabla^{M}\omega|^2 - \sum_{k=1}^{n} |A(e_{k},\omega^{\sharp})|^2 \\
		&  + \sum_{k=1}^{n}|(II(e_k,\omega^{\sharp})|^2 + |A|^2|\omega|^2 + \sum_{k=1}^{n} |II(e_k,N)|^2|\omega|^2.
		\end{align*}
		\indent Substituting in equation (\ref{eqauxpropomegaN}), we obtain
		\begin{multline*}
		\sum_{i<j}^{d}Q(u_{ij},u_{ij}) =  \quad \int_{M} \left[\sum_{k=1}^{n} |II(e_k,\omega^{\sharp})|^2 + \sum_{k=1}^{n} |II(e_k,N)|^2|\omega|^2\right] dM \\
		+ \int_{M} \left[(|\nabla^{M}\omega|^2 - \sum_{k=1}^{n} |A(e_{k},\omega^{\sharp})|^2 - Ric^{\mathcal{N}}(N,N)|\omega|^2\right]dM.
		\end{multline*}
		\indent Since $\omega$ is harmonic, integration of the Bochner-Weitzenb\"ock formula (\ref{eqbochner}) and the Gauss equation (\ref{eqGaussM}) for the minimal surface $M^n$ in $\mathcal{N}^{n+1}$ gives
		\begin{align*}
		&\int_{M} |\nabla^{M}\omega|^2 dM  \\
		& = - \int_{M} \sum_{i=1}^{n}Rm^{\mathcal{N}}(e_i,\omega^{\sharp},e_{i},\omega^{\sharp})dM + \int_{M} \sum_{i=1}^{n} |A(e_{i},\omega^{\sharp})|^2 dM.
		\end{align*}
		\indent The result follows.
	\end{proof}
	
	\begin{rmk} \label{rmkotherbetti} 
		Partially analogous computations for harmonic $p$-forms instead of one-forms also yield formulae similar to (\ref{eqomega}) and (\ref{eqomegaN}). However, there is an important difference. Whereas in the above computations all the terms involving $A$ cancel out in the final expression, for $p$-forms the final formula will generally contain terms that depend on $A$. Correspondingly, when $p\neq 1, n-1$ all index estimates one may derive by means of this approach rely on $L^2$-smallness assumptions on the second fundamental form of $M^n$, which turn out to be rather restrictive even in the very special case when the ambient manifold is the $(n+1)$-dimensional round sphere. 
	\end{rmk}
	
	\section{How the method works}
	
	\indent In this section, we show that it is possible to estimate the number of eigenvalues of the Jacobi operator of a minimal hypersurface below a certain threshold $\eta$ if there is a subspace of harmonic one-forms on this hypersurface for which the sum of the curvature terms we introduced in the previous section is bounded from above by $\eta$. We call this ``a concentration of the spectrum inequality''.
	
	\begin{prop} \label{propmain}
		Let $(\mathcal{N}^{n+1},g)$ be a Riemannian manifold isometrically embedded in $\mathbb{R}^{d}$. Let $M^{n}$ be a two-sided immersed minimal hypersurface of $\mathcal{N}^{n+1}$. Assume there exists a real number $\eta$ and a $q$-dimensional vector space $\mathcal{V}$ of harmonic one-forms on $M^{n}$ such that for every $\omega\in \mathcal{V}\setminus\left\{0\right\}$, 
			\begin{multline} \label{deshypmain}
			\int_{M} \left[\sum_{k=1}^{n} |II(e_k,\omega^{\sharp})|^2 + \sum_{k=1}^{n} |II(e_k,N)|^2|\omega|^2\right]dM \\
			- \int_{M} \left[\sum_{k=1}^{n} Rm^{\mathcal{N}}(e_k,\omega^{\sharp},e_k,\omega^{\sharp}) + Ric^{\mathcal{N}}(N,N)|\omega|^2\right] dM 
			< \eta \int_{M} |\omega|^{2}dM.
			\end{multline}
		Then
		\begin{equation*} \label{ineqmain}
		\#\{ \text{eigenvalues of the Jacobi operator of $M^{n}$ that are $< \eta$} \}  
		\geq \frac{2}{d(d-1)} q.
		\end{equation*}
	\end{prop}
	\begin{proof}
		(Compare \cite{Ros} Theorem 16 and \cite{SavA} Theorem 1.1). Let $k$ be the number of eigenvalues of the Jacobi operator (\ref{defijacobi}) of $M^{n}$ that are below $\eta$. Denote by $\phi_1,\ldots,\phi_k$ the eigenfunctions associated to the $k$ eigenvalues $\lambda_1 \leq \lambda_2\leq \lambda_3 \ldots \leq \lambda_k$ of the Jacobi operator of $M^n$ that are strictly smaller than $\eta$. \\
		\indent Fix some global orthonormal basis $\{\theta_{ij}\}_{i<j}$ of $\Lambda^{2}\mathbb{R}^d$ and let $u_{ij}=\langle N\wedge \omega^{\sharp},\theta_{ij}\rangle$ be the test functions defined in Proposition \ref{propomegaN}. The map that assigns to each $\omega\in \mathcal{V}$ the vector
		\begin{equation*}
		\left[ \int_{M}u_{ij}\phi_p dM \right],
		\end{equation*}
		\noindent where $i<j$ range from $1$ to $d$ and $p$ ranges from $1$ to $k$, is a linear map from the $q$ dimensional vector space $\mathcal{V}$ to a vector space of dimension 
		\begin{equation*}
		\binom{d}{2}k = \frac{d(d-1)}{2}k.
		\end{equation*} 
		\indent Assume, by contradiction, that $q > \frac{d(d-1)}{2}k$. Then there would exist $\omega$ in $\mathcal{V}\setminus\left\{0\right\}$ such that $\int_{M} u_{ij}\phi_{p}dM = 0$ for all $i<j$ and all $p$. Thus, as each $u_{ij}$ is $L^2$-orthogonal to all the first $k$ eigenfunctions $\phi_{p}$, from the Courant-Hilbert variational characterization of eigenvalues it follows that
		\begin{align*}
		\sum_{i<j}^{d} Q(u_{ij},u_{ij})\geq\lambda_{k+1}\sum_{i<j}^{d}\int_{M} u_{ij}^2dM = \lambda_{k+1}\int_{M} |\omega|^2dM \geq \eta\int_{M}|\omega|^2dM.
		\end{align*}
		\indent In view of Proposition \ref{propomegaN}, this is a contradiction with the assumption that inequality (\ref{deshypmain}) holds for all $\omega$ in $\mathcal{V}\setminus\left\{0\right\}$. The result follows.\\
		\indent 
	\end{proof}
	
	\begin{rmk} \label{rmkoneside}
		The same proof of the above proposition also gives more refined information when the two-sided  immersion $M^n$ arises as the two-sided cover of an one-sided immersed hypersurface in $\mathcal{N}^{n+1}$ (see sections 2.1 and 2.3). If one moreover assumes that all harmonic one-forms on the subspace $\mathcal{V}$ are such that all the corresponding functions $u_{ij}$ are \textit{odd} with respect to the deck transformation of the cover, then one can repeat the argument considering the restriction of the Jacobi operator of $M^n$ to the space of odd functions. It then follows that the number of eigenvalues of the one-sided hypersurface covered by $M^n$ below the threshold $\eta$ is bounded from below by $2q/d(d-1)$.
	\end{rmk}
	
	\indent As a corollary, in the case $\eta =0$ we obtain an estimate on the number of negative values of the Jacobi operator, that is, an index estimate:
	
	\begin{thmA}
		Let $(\mathcal{N}^{n+1},g)$ be a Riemannian manifold that is isometrically embedded in some Euclidean space $\mathbb{R}^d$. Let $M^{n}$ be a closed embedded minimal hypersurface of $(\mathcal{N}^{n+1},g)$. \\
		\indent Assume that for every non-zero vector field $X$  on $M^n$,
			\begin{multline*}
			\int_{M} \left[\sum_{k=1}^{n} Rm^{\mathcal{N}}(e_k,X,e_k,X) + Ric^{\mathcal{N}}(N,N)|X|^2\right] dM\\
			>\int_{M}\left[\sum_{k=1}^{n} |II(e_k,X)|^2 + \sum_{k=1}^{n} |II(e_k,N)|^2|X|^2\right]dM.
			\end{multline*}
		\indent Then
		\begin{equation} \label{ineqcor}
		index(M) \geq \frac{2}{d(d-1)}b_{1}(M).
		\end{equation} 
	\end{thmA}
	\begin{proof}
		Assume $M^n$ is two-sided. Under the assumption of the corollary, the hypothesis of Proposition \ref{propmain} is automatically satisfied for $\eta=0$ and $\mathcal{V}$ the set of all harmonic one-forms on $M^n$, whose dimension is $b_{1}(M)$. Inequality (\ref{ineqcor}) follows. \\
		\indent When $M^n$ is one-sided, let $\pi = \hat{M}^n \rightarrow M^n \subset\mathcal{N}^{n+1}$ be its two-sided cover and let $\mathcal{V}$ be the set of harmonic one-forms on $\hat{M}^n$ that are \textit{invariant} under the deck transformation $\tau : \hat{M}^n \rightarrow \hat{M}^n$. This space has dimension at least $b_{1}(M)$, as it contains all the forms $\pi^{*}\omega$, where $\omega$ is harmonic on $M^n$ (in fact, $\pi^{*}: H^{1}(M;\mathbb{R}) \rightarrow H^{1}(\hat{M};\mathbb{R})$ is injective). The result follows as a consequence of Proposition \ref{propmain} (see Remark \ref{rmkoneside}) once one checks that, by construction of $\hat{M}^n$, for all $\hat{\omega}$ in $\mathcal{V}$, each function $\hat{u}_{ij}=\langle\hat{N}\wedge \omega^{\sharp} , \theta_{ij}\rangle$ satisfy
		\begin{equation*}
		\hat{u}_{ij}(\tau(x)) = \langle\hat{N}(\tau(x))\wedge \omega^{\sharp}(\tau(x)) , \theta_{ij}\rangle = \langle-\hat{N}(x)\wedge \omega^{\sharp}(x) , \theta_{ij}\rangle = - \hat{u}_{ij}(x) 
		\end{equation*}
		\noindent for all $x$ in $\hat{M}^n$, i.e, all functions $u_{ij}$ are \textit{odd} with respect to the deck transformation $\tau$.
	\end{proof}
	
	\indent It is possible to show an exactly analogous concentration of the spectrum inequality under a hypothesis that is compatible with Proposition \ref{propomega}. Instead, by combining formula (\ref{eqauxpropomega}) for both the coordinates of the one-form $\omega$ and the coordinates of its Hodge dual $*\omega$, we shall prove a slightly better result, in the sense that the pinching assumption is weaker (and thus easier to verify in applications).
	
	\begin{prop} \label{propmainb}
		Let $(\mathcal{N}^{3},g)$ be a Riemannian manifold isometrically embedded in $\mathbb{R}^{d}$. Let $M^{2}$ be a closed oriented embedded minimal surface of $\mathcal{N}^{3}$. Assume there exists $\eta$ and a $q$-dimensional vector space $\mathcal{V}$ of harmonic one-forms on $M^{n}$ such that for every $\omega$ in $\mathcal{V}\setminus\left\{0\right\}$,
			\begin{equation} \label{deshypmainb}
			\int_{M} \sum_{k=1}^{2} \left(|II(e_k,\omega^{\sharp})|^2 + |II(e_k,*\omega^{\sharp})|^2\right) - R^{\mathcal{N}}|\omega|^2 dM < 2\eta \int_{M} |\omega|^{2}dM.
			\end{equation}
		Then
		\begin{equation} \label{ineqmainb}
		\#\{ \text{eigenvalues of the Jacobi operator of $M^{n}$ that are $< \eta$} \} \geq \frac{1}{2d} q.
		\end{equation}
	\end{prop}
	\begin{proof}
		\indent Let $k$ denote the number of eigenvalues of $J_{M}$ that are strictly less than $\eta$. Keeping the notations of the proof of Proposition \ref{propmain} and fixing an orthonormal basis $\left\{\theta_i\right\}$ of $\Lambda^1\mathbb{R}^d$, consider the linear map that assigns to each harmonic one-form $\omega$ in $\mathcal{V}$ the matrix
		\begin{equation*}
		[\int_{M} u_i\phi_j dM, \int_{M} u^{\ast}_i\phi_j dM],
		\end{equation*}
		\noindent which belongs to a $2dk$-dimensional real vector space, for $u_i=\langle\omega^{\sharp},\theta_i\rangle$ and $u^{\ast}_i=\langle*\omega^{\sharp},\theta_{i}\rangle $.\\
		\indent If $q > 2dk$ there would exist some non-trivial harmonic one-form $\omega$ such that the coordinates of both $\omega$ and $*\omega$ in $\mathbb{R}^{d}$ would be orthogonal to all the first $k$ eigenfunctions $\phi_1,\ldots,\phi_k$. But then
		\begin{equation*}
		\sum_{i=1}^{d}Q(u_i,u_i) + \sum_{i=1}^{d}Q(u^{*}_i,u^{*}_{i}) \geq 2\lambda_{k+1} \int_{M} |\omega|^{2}dM \geq 2\eta \int_{M} |\omega|^{2}dM.
		\end{equation*}
		\indent In view of Proposition \ref{propomega}, this is a contradiction with hypothesis \eqref{deshypmainb}. Therefore $q \leq 2dk$, as we wanted to prove.
	\end{proof}
	
	\section{Some applications}
	
	\indent We present a gallery of examples of ambient manifolds $(\mathcal{N}^{n},g)$ for which our general computations yield, for all of its closed embedded minimal hypersurfaces $M^{n}$, the conjectured lower bound of the index by the first Betti number. As it is clear from the formulae in Section 3, the success of the method depends not only on the intrinsic geometry of $(\mathcal{N}^{n},g)$, but also of the choice of some isometric embedding of it in some Euclidean space. Our examples are such that there is either an obvious or a ``most beautiful'' choice for this embedding.
	
	\subsection{Round spheres (after Savo)} 
	
	\begin{thm}[Savo] \label{thmsavo}
		Let $M^{n}$ be a closed embedded minimal hypersurface of the unit sphere $S^{n+1}$ in $\mathbb{R}^{n+2}$. If $M^n$ is not totally geodesic, then
		\begin{equation*}
		index(M) \geq \frac{2}{(n+2)(n+1)}b_{1}(M) + n + 2.
		\end{equation*}
	\end{thm}
	\begin{proof}
		The unit sphere in $\mathbb{R}^{n+2}$ has constant sectional curvature equal to one and is totally umbilic. In fact, its second fundamental form $II$ satisfies $|II(X,Y)|=|\langle X,Y\rangle|$ for all tangent vector fields $X$ and $Y$. \\
		\indent Thus, for any closed embedded minimal hypersurface $M^{n}$ of $S^{n+1}$, it is immediate to check that for any harmonic one-form $\omega$ on $M^{n}$,
		\begin{multline*}
		 \sum_{k=1}^{n}|II(e_k,\omega^{\sharp})|^2 + \sum_{k=1}^{n}|II(e_k,N)|^2|\omega|^{2} \\
		 - \sum_{k=1}^{n} Rm^{\mathcal{N}}(e_k,\omega^{\sharp},e_k,\omega^{\sharp}) - Ric(N,N)|\omega|^2 \\ 
		= -(2n-2) |\omega|^2. 
		\end{multline*}
		\indent By Proposition \ref{propmain}, we conclude that
		\begin{multline*}
		 \#\{ \text{eigenvalues of $J_M$ that are less or equal than $-2n+2$}  \} \\
		  \geq \frac{2}{(n+2)(n+1)} b_{1}(M).
		\end{multline*}
		\indent Moreover, it is possible to check that each coordinate of the unit normal field $N=(N_1,\ldots,N_{n+2})$ of $M^{n}$ satisfies the equation
		\begin{equation*}
		J_{M}N_i - n N_{i} = \Delta_{M}N_i + |A|^2 N_{i} = 0,
		\end{equation*} 
		\noindent and also to show that, when $M^n$ is not totally geodesic, the multiplicity of $-n$ as an eigenvalue of $J_M$ is at least $n+2$ (this dates back to J. Simons \cite{Sim}, see for example \cite{SavA}, proof of Corollary 2.2). Thus, when $n\geq 3$ we can estimate the index of any embedded minimal non-totally geodesic hypersurface $M^{n}$ of the round sphere $S^{n+1}$ by
		\begin{equation*}
		index(M) \geq \frac{2}{(n+2)(n+1)}b_{1}(M) + n + 2.
		\end{equation*}
		The case $n=2$ is somehow peculiar and an \textsl{ad hoc} argument is needed to conclude, for which we refer the reader to Section 4 of \cite{SavA}.	\end{proof}
	
	\begin{rmk}
		In ambient dimension three, Savo proved a better estimate by using formula \eqref{deshypmainb} for the coordinates of harmonic one-forms and by performing a detailed analysis of the dimension of the set of harmonic one-forms whose coordinates all belong to the eigenspace of the Jacobi operator associated to the eigenvalue $-2$ (see the proof of Theorem 1.3 in \cite{SavA}).
	\end{rmk}
	
	\
	
	\subsection{Real projective spaces} \label{subs:realproj}
	
	\begin{thm}
		Let $M^{n}$ be a closed embedded minimal hypersurface of the real projective space $\mathbb{RP}^{n+1}$ endowed with its metric of constant sectional curvature one. Then
		\begin{equation*}
		index(M) \geq \frac{2}{(n+2)(n+1)}b_{1}(M).
		\end{equation*}
	\end{thm}
	\begin{proof}
		\indent There is an obvious one-to-one correspondence between immersed minimal hypersurfaces $M^n$ of $\mathbb{RP}^{n+1}$ and immersed minimal hypersurfaces $\tilde{M}^{n}$ of $S^{n+1}$ that are invariant under the antipodal map (all embedded minimal hypersurfaces of $S^{n+1}$ are connected). $\tilde{M}^n$ is a two-to-one cover of $M^n$. \\
		\indent If $M^n$ is a closed embedded two-sided hypersurface in $\mathbb{RP}^{n+1}$ with unit normal field $N$ (or is the two-sided cover of a closed embedded one-sided minimal hypersurface), then $\tilde{M}^{n}$ is two-sided in $S^{n+1}$ with unit normal field $\tilde{N}$ such that 
		\begin{equation*}
		\tilde{N}(-x) = -\tilde{N}(x) \quad \text{for all} \quad x\in \tilde{M}^{n}. 
		\end{equation*}
		\indent Given any harmonic one-form $\omega$ on $M^n$, its pull-back $\tilde{\omega}$ in $\tilde{M}^{n}$ is also harmonic, and such that $\tilde{\omega}(-x) = -\tilde{\omega}(x)$ for all $x\in \tilde{M}^n$. For these forms, the coodinate functions
		\begin{equation*}
		\tilde{u}_{ij} = \langle\tilde{N}\wedge \tilde{\omega}^{\sharp}, \theta_{ij}\rangle, \quad i,j = 1,\ldots n+2,\, i<j,
		\end{equation*}
		\noindent satisfy 
		\begin{equation*}
		\tilde{u}_{ij}(-x) = \tilde{u}_{ij}(x).
		\end{equation*}
		\indent Therefore there are well-defined functions $u_{ij}$ on $M^n$ whose lifts to $\tilde{M}^n$ are precisely $\tilde{u}_{ij}$. Since the projection of $S^{n+1}$ on $\mathbb{RP}^{n+1}$ is a local isometry, when the functions $u_{ij}$ are plugged in the index form of $M^n$ one obtains the same as when the functions $\tilde{u}_{ij}$ are plugged in the index form of $\hat{M}^n$, that is
		\begin{equation*}
		\sum_{i<j}^{n+2}Q^{M}(u_{ij},u_{ij}) = -(2n-2)\int_{M}|\omega|^2dM < 0
		\end{equation*}
		\noindent by the computations in Theorem \ref{thmsavo}. Applying the general method (as in the proof of Theorem A), we conclude that every embedded minimal hypersurface of $\mathbb{RP}^{n+1}$ is such that
		\begin{equation*}
		index(M) \geq \frac{2}{(n+2)(n+1)} b_{1}(M).
		\end{equation*}
		\indent
	\end{proof}
	
	\begin{rmk}
		It is not possible to obtain an extra $n+2$ in the lower bound for the index as in the case of spheres because $\tilde{N}$ is such that $\tilde{N}(-x)=-\tilde{N}(x)$ for all $x$ in $\tilde{M}^{n}$, which means that the coordinates of $\tilde{N}$ do not descend to $M^n$ to produce test functions for the index form.
	\end{rmk}
	
	\subsection{Complex and quaternionic projective spaces, and the Cayley plane} \label{subs:rankone}
	
	In this section we want to consider $(\mathcal{N}^{n+1},g)$ to be one of the projective spaces:
	\begin{equation*}
	\mathcal{N}^{n+1}  = \mathbb{CP}^{m} \, (2m=n+1), \, \mathbb{HP}^{p} \, (4p=n+1), \, Ca\mathbb{P}^2 \, (16=n+1),
	\end{equation*} 
	\noindent endowed with their standard Riemannian metrics that are symmetric, i.e., whose Riemman curvature tensor is parallel. In the case of $\mathbb{CP}^m$ this is just the standard Fubini-Study metric. \\
	\indent Up to normalization, all these metrics have sectional curvatures bounded between $1$ and $4$. Moreover, each of them is Einstein. The following table summarizes the relevant information about the curvature of these spaces:
	\vspace{0.5cm}
	\begin{center}
		\begin{tabular}{ | c | c | c | } 
			\hline 
			$\mathcal{N}^{n+1}$ & $Rm^{\mathcal{N}}(X,Y,X,Y)$, $X,Y$ orthonormal  & $Ric^{\mathcal{N}}$  \\ \hline
			$\mathbb{CP}^{m}$   &                $1 + 3g(X,JY)^2$                     &  $(n+3)g$   \\ \hline
			$\mathbb{HP}^{p}$   &     $1 + 3g(X,IY)^2 + 3g(X,JY)^2 + 3g(X,KY)^2$      &  $(n+9)g$   \\ \hline
			$Ca\mathbb{P}^2$    &             bounded between 1 and 4                 &   $36g  $   \\ \hline
		\end{tabular}
	\end{center}
	\vspace{0.5cm}
	
	\indent In the above formulae for the sectional curvatures, $J$ denote the compatible complex structure of $\mathbb{CP}^m$ and $I,J,K$ denote the compatible complex structures of $\mathbb{HP}^{m}$ that satisfy the standard operational rules of the quaternions. \\
	\indent There exists a beautiful family of isometric embeddings of these spaces into an  Euclidean space. These embeddings generalise the standard Veronese embeddings of $\mathbb{RP}^n$ in the space of symmetric $(n+1)\times (n+1)$ matrices over $\mathbb{R}$. We refer the reader to \cite{CheA} (and the references threrein) for a detailed account of these embeddings, which enjoy many other nice geometric properties.\\
	\indent Following \cite{CheA}, let us give a brief description of these embeddings. Denote by $\mathbb{F}$ either the field of complex numbers $\mathbb{C}$ or the division algebra of quaternions $\mathbb{H}$. Consider the action of the set of unit elements in $\mathbb{F}$ on the unit sphere in $\mathbb{F}^{m+1}$ by multiplication (from the right). The projective spaces $\mathbb{FP}^m$ are the quotient spaces obtained by identifying points in the same orbit of this action. \\
	\indent  Let $M(m+1;\mathbb{F})$ denote the set of all $(m+1)\times (m+1)$ matrices with coefficients in $\mathbb{F}$. A Hermitian matrix is a matrix $A$ in $M(m+1;\mathbb{F})$ that coincides with its transpose conjugate, $\overline{A}^{t}$. The set $H(m+1;\mathbb{F})$ of all Hermitian matrices can be seen as an Euclidean space when endowed with the metric 
	\begin{equation*}
	\langle A,B\rangle =  \frac{1}{2}Re(tr (AB)).
	\end{equation*}
	\indent This space has dimension $(m+1)^2$ if $\mathbb{F}=\mathbb{C}$ or $(2m+1)(m+1)$ if $\mathbb{F}=\mathbb{H}$. \\
	\indent The map 
	\begin{equation*}
	[z_{1} : \ldots : z_{m+1}] \in \mathbb{FP}^{m} \mapsto \left[ z_{i}\overline{z}_{j} \right] _{ij} \in H(m+1;\mathbb{F})
	\end{equation*}
	\noindent is well-defined and gives an embedding of $\mathbb{FP}^m$ into $H(m+1;\mathbb{F})$. The image of this map is the set 
	\begin{equation*}
	\{ A \in H(m+1;\mathbb{F});\, A^{2} = A, tr A = 1\},
	\end{equation*}
	\noindent and the induced metric is precisely the canonical metric on $\mathbb{FP}^{m}$ described above. \\
	\indent The case of the Cayley plane $Ca\mathbb{P}^2$ has a similar algebraic characterisation (see \cite{CheA}). It can be seen as an embedded hypersurface of an Euclidean space of dimension 27. \\
	\indent For the discussion that follows, the main property we need to know, and that is shared by all of these embeddings, is that the second fundamental form $II$ satisfies
	\begin{equation}\label{eqemb}
	\langle II(X,X),II(X,X)\rangle = 4 \quad \text{for all unit} \quad X\in \mathcal{X}(\mathcal{N}),
	\end{equation}
	\noindent see equations (3.2) and (5.1) in \cite{CheA}. It follows by polarization of this identity that for every pair of orthogonal unit vectors fields $X$ and $Y$ that are tangent to $\mathcal{N}^{n+1}$,
	\begin{equation} \label{equm}
	\langle II(X,X),II(Y,Y)\rangle + 2|II(X,Y)|^2 = 4.
	\end{equation}
	\indent Combining (\ref{equm}) and the Gauss equation (\ref{eqGaussN}), one obtains the following formula, valid for all pairs of orthogonal vectors $X$ and $Y$ that are tangent to $\mathcal{N}^{n+1}$:
	\begin{equation} \label{eqdois}
	|II(X,Y)|^2 = \frac{1}{3}(4- Rm^{\mathcal{N}}(X,Y,X,Y)).
	\end{equation}
	\indent Given $\omega$ a harmonic one-form on a given closed embedded oriented minimal hypersurface $M^n$ of $\mathcal{N}^{n+1}$, it is possible to compute that, at any given point $p$ in $M^{n}$, if $\{e_k\}$ is an arbitrary orthonormal basis of $T_{p}M$,
	\begin{equation} \label{eqauxp1}
	\left(\sum_{k=1}^{n}|II(e_k,N)|^2 - Ric^{\mathcal{N}}(N,N)\right)|\omega|^2 = \frac{4}{3}( n - K)|\omega|^2
	\end{equation}
	\noindent where $K$ is the Einstein constant of $(\mathcal{N}^{n+1},g)$ given in third column of Table 1. Moreover, at points where $\omega$ does not vanish it is possible to choose an orthonormal basis $\{e_k\}$ of $T_pM$ such that $e_1=\omega^{\sharp}/|\omega^{\sharp}|$. Thus, by \eqref{eqemb} and \eqref{eqdois}
	\begin{multline} \label{eqauxp2}
	\sum_{k=1}^{n}\left(|II(e_k,\omega^{\sharp})|^2 - Rm^{\mathcal{N}}(e_k,\omega^{\sharp},e_k,\omega^{\sharp})\right) \\
	= \left(|II(e_1,e_1)|^2 + \frac{4}{3}\sum_{k=2}^{n}(1-Rm^{\mathcal{N}}(e_k,e_1,e_k,e_1)\right) |\omega|^2 \\
	= \frac{4}{3}(n+2 - K + Rm^{\mathcal{N}}(N,e_1,N,e_1))|\omega|^2.
	\end{multline}
	\indent Since $Rm^{\mathcal{N}}(N,e_1,N,e_1)$ is bounded between $1$ and $4$, from (\ref{eqauxp1}) and (\ref{eqauxp2}) it follows that 
	\begin{multline*}
	 \sum_{k=1}^{n}|II(e_k,\omega^{\sharp})|^2 + \sum_{k=1}^{n}|II(e_k,N)|^2|\omega|^{2} \\
	 - \sum_{k=1}^{n} Rm^{\mathcal{N}}(e_k,\omega^{\sharp},e_k,\omega^{\sharp}) - Ric^{\mathcal{N}}(N,N)|\omega|^2 \\  \leq  \frac{8}{3}(n+3 - K)|\omega|^2.
	\end{multline*}
	\indent Notice that equality holds if and only if $Rm^{\mathcal{N}}(N,\omega^{\sharp},N,\omega^{\sharp}) = 4|\omega|^2$ on all points of $M^{n}$. Reading the values of $K$ from the above Table, it is then immediate to check that 
	\begin{multline*}
	\int_{M} \left[\sum_{k=1}^{n}|II(e_k,\omega^{\sharp})|^2 + \sum_{k=1}^{n}|II(e_k,N)|^2|\omega|^{2}\right]dM \\ - \int_{M} \left[\sum_{k=1}^{n} Rm^{\mathcal{N}}(e_k,\omega^{\sharp},e_k,\omega^{\sharp}) + Ric^{\mathcal{N}}(N,N)|\omega|^2\right] dM \leq 0,
	\end{multline*}
	\noindent where equality can only happen when $\mathcal{N}^{n+1}$ is the complex projective plane $\mathbb{CP}^{m}$ and the harmonic form $\omega$ on $M^{n}$ is such that $Rm^{\mathcal{N}}(N,\omega^{\sharp},N,\omega^{\sharp}) = |\omega|^2+3g(\omega^{\sharp},JN)^2 = 4|\omega|^2$. In turn, this implies at once (by virtue of the Cauchy-Schwartz inequality) that there exists a smooth function $f:M\to\mathbb{R}$ such that $\omega^{\sharp} = fJN.$ We shall see in Appendix \ref{app:borderline} that, in this setting, necessarily $f=0$ identically on $M^n$ (see Proposition A.0 for a precise statement), and hence $\omega$ must be the trivial one-form. As a result, the inequality in question must be strict for all non-zero harmonic forms. \\

	\indent Combining all those facts, we have shown that it is possible to apply Theorem A to obtain the following results.
	
	\begin{thm}\label{thm:compl}
		Let $M^{n}$ be a closed embedded minimal hypersurface of the complex projective space $\mathbb{CP}^{m}$, $2m=n+1$. Then
		\begin{equation*}
		index(M) \geq \frac{2}{m(m+2)(m+1)^{2}} b_{1}(M). 
		\end{equation*}
	\end{thm}
	
	\begin{thm}\label{thm:quat}
		Let $M^{n}$ be a closed embedded minimal hypersurface of the quaternionic projective space $\mathbb{HP}^{p}$, $4p=n+1$. Then
		\begin{equation*}
		index(M) \geq \frac{2}{(2p+3)(2p+1)(p+1)p} b_{1}(M).
		\end{equation*}

	\end{thm}
	
	\begin{thm}\label{thm:Cay}
		Let $M^{n}$ be a closed embedded minimal hypersurface of the Cayley plane $Ca\mathbb{P}^{2}$. Then
		\begin{equation*}
		index(M) \geq \frac{1}{351} b_{1}(M).
		\end{equation*}
	\end{thm}
	
	\subsection{Product of the circle and unit spheres}\label{subs:circle} In this subsection, we consider $(\mathcal{N}^{n+1},g)$ to be the Riemannian product of the unit circle $S^1$ with the unit round sphere $S^{n}$. As these spaces have non-negative Ricci curvature, this case goes beyond the situations described in the conjecture stated in the Introduction. However, it is still possible to obtain similar bounds for the Morse index of their embedded hypersurfaces. \\
	\indent The case of closed embedded minimal surfaces in $S^{1}\times S^2$ was already studied in detail by Urbano in \cite{Urb2} (see  Theorem 4.8 in \cite{Urb2} for the precise statement of his result). We remark that the proof in \cite{Urb2} used the coordinates of harmonic one-forms as test functions for the index form, i.e., the proof used formula (\ref{eqauxpropomega}) specialized to the case of the canonical product embedding of $S^{1}\times S^{2}$ in $\mathbb{R}^5$. \\
	\indent Thus, in the sequel we restrict ourselves to higher dimensions.
	\begin{thm} \label{thmprod}
		Let $M^{n}$ be a closed embedded minimal hypersurface of $S^{1}\times S^{n}$, $n\geq 3$. Then
		\begin{equation*}
		index(M) \geq \frac{2}{(n+3)(n+2)} b_{1}(M).
		\end{equation*}
	\end{thm} 
	\begin{proof}
		We consider the standard embedding of $\mathcal{N}^{n+1}=S^{1}\times S^{n}$ into $\mathbb{R}^{n+3}=\mathbb{R}^2\times \mathbb{R}^{n+1}$. It is a product embedding of totally umbilic hypersurfaces. The tangent space of $\mathcal{N}^{n+1}$ decomposes as $TS^1\oplus TS^{n}$. Let $\pi_{1}: T\mathcal{N} \rightarrow TS^1$, $\pi_{2}:T\mathcal{N} \rightarrow TS^n$ denote the corresponding projections. \\
		\indent Since $S^1$ is one-dimensional and $S^n$ has the constant sectional curvatures equal to one, the Riemann curvature tensor of $(\mathcal{N}^{n+1},g)$ is given by
		\begin{equation*}
		Rm^{\mathcal{N}}(X,Y,X,Y) = |\pi_{2}(X)|^2||\pi_{2}(Y)|^2 - \langle\pi_{2}(X),\pi_{2}(Y)\rangle^2
		\end{equation*}
		\noindent for all $X,Y\in \mathcal{X}(\mathcal{N})$. In particular, 
		\begin{equation*}
		Ric^{\mathcal{N}}(X,X) = (n-1)|\pi_{2}(X)|^2 \quad \text{for all} \quad X\in \mathcal{X}(\mathcal{N})
		\end{equation*}
		\indent The second fundamental form of $\mathcal{N}^{n+1}$ in $\mathbb{R}^{n+3}$ is given by
		\begin{align*}
		|II(X,Y)|^2 & = \langle\pi_{1}(X),\pi_{1}(Y)\rangle^2 + \langle\pi_{2}(X),\pi_{2}(Y)\rangle^2 \\
		&  = \langle X,\pi_{1}(Y)\rangle^2 + \langle X,\pi_{2}(Y)\rangle^2.
		\end{align*}
		\indent Let $M^n$ be an oriented minimal hypersurface $M^n$ of $\mathcal{N}^{n+1}$ (or the two-sided cover of a  one-sided hypersurface) and $\omega$ be a harmonic one-form on $M^n$. On one hand, we have
		\begin{align} \label{prodeqA}
		\nonumber \sum_{k=1}^{n} |II(e_{k},N)|^2 & = \sum_{k=1}^{n}\langle e_{k},\pi_{1}(N)\rangle^2 + \sum_{k=1}^{n}\langle e_{k},\pi_{2}(N)\rangle^2 \\
		\nonumber                                & = |\pi_{1}(N)|^2 - \langle N,\pi_{1}(N)\rangle^2 + |\pi_{2}(N)|^2 - \langle N,\pi_{2}(N)\rangle ^2 \\
		\nonumber                                & = |\pi_{1}(N)|^2-|\pi_{1}(N)|^4 + |\pi_{2}(N)|^2-|\pi_{2}(N)|^4 \\
		& = 1 - |\pi_{1}(N)|^4 - |\pi_{2}(N)|^4 ,
		\end{align}
		\noindent and, analogously,
		\begin{equation} \label{prodeqB}
		\sum_{k=1}^{n} |II(e_{k},\omega^{\sharp})|^2  = |\omega|^2 - \langle N,\pi_{1}(\omega^{\sharp})\rangle^2 - \langle N,\pi_{2}(\omega^{\sharp})\rangle^2.
		\end{equation}
		\indent On the other hand, a similar computation gives
		\begin{align} \label{prodeqC}
		\nonumber \sum_{k=1}^{n}Rm^{\mathcal{N}}(e_{k},\omega^{\sharp},e_{k},\omega^{\sharp})  & =
		\sum_{k=1}^{n}|\pi_{1}(e_{k})|^2|\pi_{1}(\omega^{\sharp})|^2 + \sum_{k=1}^{n}|\pi_{2}(e_{k})|^2|\pi_{2}(\omega^{\sharp})|^2 \\
		& \quad + \langle N,\pi_{1}(\omega^{\sharp})\rangle^2 + \langle N,\pi_{2}(\omega^{\sharp})\rangle^2 - |\omega^{\sharp}|^2. 
		\end{align}
		\indent In order to proceed with the computation, we choose a convenient orthonormal basis of $T_{p}\mathcal{N}$ at some point $p=(p_{1},p_{2})$ in $S^{1}\times S^n$. Let $w$ denote the unit vector field on $S^1$. The set of vectors on $T_{p_2}S^{n}$ that are orthogonal to $N$ constitute a vector subspace of dimension at least $n-1$. Therefore we can choose an orthonormal basis for $T_{p}\mathcal{N}$ of the form $\{T,f_{1},f_{2},\ldots,f_{n-1},N\}$, where $\{f_{1},f_{2},\ldots,f_{n-1},f_{n}\}$ is an orthonormal basis of $T_{p_{2}}S^{n}$ and the vectors $T$ and $N$ can be written as
		\begin{align*}
		T & = \cos(\theta) w + \sin(\theta) f_{n} \\
		N & = -\sin(\theta)w + \cos(\theta) f_{n}.
		\end{align*}
		\indent  Furthermore, if we let $\varphi\in[0,\pi]$ be the angle between the vectors $\omega^{\sharp}$ and $T$ in $T_{p}M$ we can express (\ref{prodeqA}), (\ref{prodeqB}), (\ref{prodeqC}) and $Ric^{\mathcal{N}}(N,N)$ in terms of the angles $\theta$ and $\varphi$, since in particular 
		\begin{align*}
		|\pi_1(\omega^{\sharp})|^2 &=\cos^2(\theta)\cos^2(\varphi)|\omega|^2, \	|\pi_2(\omega^{\sharp})|^2 =(\sin^2(\varphi)+\sin^2(\theta)\cos^2(\varphi))|\omega|^2
		\end{align*}
		as well as
		\begin{align*}
		\langle\pi_1(N),\pi_1(\omega^{\sharp})\rangle^2 = 
		\langle\pi_2(N),\pi_2(\omega^{\sharp})\rangle^2= \cos^2(\theta)\sin^2(\theta)\cos^2(\varphi)|\omega|^2.
		\end{align*}
		\indent Hence, combining these equations we obtain the following formula (for the \textsl{opposite} of the integrand in \eqref{eqomegaN}):
		\indent 
		\begin{align*}
		& \sum_{k=1}^{n}Rm^{\mathcal{N}}(e_{k},\omega^{\sharp},e_{k},\omega^{\sharp})+  \sum_{k=1}^{n}Rm^{\mathcal{N}}(e_k,N,e_k,N)|\omega|^2\\
		&-\sum_{k=1}^{n}|II(e_k,\omega^{\sharp})|^2-\sum_{k=1}^{n}|II(e_k,N)|^2|\omega|^2 \\
		= & \left((n-1)|\pi_2(N)|^2+|\pi_1(N)|^4+|\pi_2(N)|^4-3\right)|\omega|^2\\
		 &+|\pi_{1}(\omega^{\sharp})|^2 \sum_{k=1}^{n}|\pi_{1}(e_{k})|^2 + |\pi_{2}(\omega^{\sharp})|^2\sum_{k=1}^{n}|\pi_{2}(e_{k})|^2 \\
		  &+2\langle \pi_{1}(N),\omega^{\sharp}\rangle^2+ 2\langle\pi_{2}(N),\omega^{\sharp}\rangle^2 \\
		  = & ((n-1)\cos^2(\theta)+\cos^4(\theta)+\sin^4(\theta)-3)|\omega|^2 \\
			&+ (\cos^2(\theta))|\pi_1(\omega^{\sharp})|^2+ ((n-1)+\sin^2(\theta))|\pi_2(\omega^{\sharp})|^2 \\
			&+4\cos^2(\theta)\sin^2(\theta)\cos^2(\varphi)|\omega|^2 \\
			=&((n-1)\cos^2(\theta)+\cos^4(\theta)+\sin^4(\theta)-3)|\omega|^2\\
			&+(\cos^2(\theta)+2\sin^2(\theta))|\pi_1(\omega^{\sharp})|^2 +((n-1)+\sin^2(\theta))|\pi_2(\omega^{\sharp})|^2 \\
			&+2\cos^2(\theta)\sin^2(\theta)\cos^2(\varphi)|\omega|^2 \\
			=& ((n-2)\cos^2(\theta)+\cos^4(\theta)+\sin^4(\theta)-1+2\cos^2(\theta)\sin^2(\theta)\cos^2(\varphi))|\omega|^2\\
			&+(n-2)(\sin^2(\varphi)+\sin^2(\theta)\cos^2(\varphi))|\omega|^2 \\
			=& (n-2)(\cos^2(\theta)+\sin^2(\theta)\cos^2(\varphi)+\sin^2(\varphi))|\omega|^2\\
			&+(\cos^4(\theta)+\sin^4(\theta)-1+2\cos^2(\theta)\sin^2(\theta)\cos^2(\varphi)) |\omega|^2.
			\end{align*}
		\indent As a result, using the fact that $n\geq 3$ it is sufficient for us to prove that when $\varphi,\theta\in[0,\pi]$
		\begin{multline*}
		q(\theta,\varphi)=\cos^2(\theta)+\sin^2(\theta)\cos^2(\varphi)+\sin^2(\varphi)\\+\cos^4(\theta)+\sin^4(\theta)-1+2\cos^2(\theta)\sin^2(\theta)\cos^2(\varphi)>0.
		\end{multline*}
		\indent Replacing, in the expression for $q$, $\cos^2(\varphi)=1-\sin^2(\varphi)$ and making use of basic trigonometric identities we easily see that in fact
		\begin{align*}
		q(\theta,\varphi)=1+\sin^2(\varphi)\cos^2(\theta)(2\cos^2(\theta)-1)
		\end{align*}
		so that patently $q\geq 1$ for $\theta\in [0,\pi/4]\cup[3\pi/4,\pi]$ while for $\theta\in[\pi/4,3\pi/4]$
		\begin{align*}
		q(\theta,\varphi)\geq 1+\cos^2(\theta)(2\cos^2(\theta)-1)= 2\cos^4(\theta)-\cos^2(\theta)+1 \geq \frac{7}{8}.
		\end{align*}

		 By Theorem A, the proof is complete.
	\end{proof}
	
	In fact, as we anticipated in the Introduction, a similar approach allows to handle the case when the ambient manifold is the product of two round spheres of arbitrary dimension (with the sole exception of $S^2\times S^2$). 
	
	\begin{thm} \label{thmspheres}
		Let $M^{n}$ be a closed embedded minimal hypersurface of $S^{p}\times S^{q}$, $p,q\geq 2$ and $(p,q)\neq (2,2)$. Then
		\begin{equation*}
		index(M) \geq \frac{2}{(p+q+2)(p+q+1)} b_{1}(M).
		\end{equation*}
	\end{thm}
	
	For the sake of brevity, we decided to omit the somewhat lenghty proof of this assertion, which follows along the very same lines of the argument we presented for $S^1\times S^n$ with complications of purely notational character.

	\subsection{Pinched convex hypersurfaces of the Euclidean space}
	
	\begin{thm} \label{thmellip}
		Let $(\mathcal{N}^{n+1},g)$ be a closed embedded hypersurface of $\mathbb{R}^{n+2}$ whose principal curvatures $k_1 \leq \ldots \leq k_{n+1}$ with respect to the outward pointing unit normal $\nu$ are positive and satisfy the following pinching condition:
		\begin{equation*}
		\frac{k_{n+1}}{k_{1}} < \sqrt{\frac{n+1}{2}}.
		\end{equation*}
		Then, every closed embedded minimal hypersurface $M^n$ of $\mathcal{N}^{n+1}$ is such that
		\begin{equation*}
		index(M) \geq \frac{2}{(n+2)(n+1)} b_{1}(M).
		\end{equation*}
	\end{thm}
	\begin{proof}
		\indent It is convenient to consider the shape operator $S$ of $\mathcal{N}^{n+1}$, i.e., the endomorphism of the tangent bundle of $\mathcal{N}^{n+1}$ defined by 
		\begin{equation*}
		S : X \in \mathcal{X}(\mathcal{N}) \mapsto -\nabla_{X}\nu \in \mathcal{X}(\mathcal{N}).
		\end{equation*}
		\indent $S$ and $II$ are related by $II(X,Y) = \langle S(X),Y\rangle\nu$. Using the Gauss equation (\ref{eqGaussN}) for $\mathcal{N}^{n+1}$, we can rewrite the integrand appearing in formula \eqref{eqomegaN} only in terms of $S$. Thereby, one obtains
		\begin{multline*}
		 \sum_{k=1}^{n} |II(e_k,\omega^{\sharp})|^2 + \sum_{k=1}^{n} |II(e_k,N)|^2|\omega|^2 
		\\ -  \sum_{k=1}^{n} Rm^{\mathcal{N}}(e_k,\omega^{\sharp},e_k,\omega^{\sharp}) - Ric^{\mathcal{N}}(N,N)|\omega|^2  \\ = -  \left(\langle S(\omega^{\sharp}),\omega^{\sharp}\rangle + \langle S(N),N\rangle|\omega|^2\right)\sum_{k=1}^{n} \langle S(e_k),e_k\rangle  \\ + 2 \sum_{k=1}^{n} \left(\langle S(\omega^{\sharp}),e_k\rangle^2 + \langle S(N),e_k\rangle^2|\omega|^2\right). 
		\end{multline*}
		\indent Choosing an orthonormal frame $\{e_{1},\ldots,e_{n},N\}$ on $\mathcal{N}^{n+1}$ and observing that \begin{equation*}
		k_{1}|X|^2 \leq \,\, \langle S(X),X\rangle \,\, \leq k_{n+1}|X|^2 \quad \text{for all} \quad X\in \mathcal{X}(\mathcal{N}),
		\end{equation*}
		\noindent it is possible to derive the following estimates:
		\begin{align*}
		\sum_{k=1}^{n} \langle S(\omega^{\sharp}),e_k\rangle ^2 = |S(\omega^{\sharp})|^2 - \langle S(\omega^{\sharp}),N\rangle^2 & \leq k_{n+1}^{2}|\omega|^2 \\
		\sum_{k=1}^{n} \langle S(N),e_k\rangle^2|\omega|^2  = (|S(N)|^2 - \langle S(N),N\rangle^2)|\omega|^2 & \leq (k_{n+1}^2-k_{1}^{2})|\omega|^ 2 \\
		\left(\langle S(\omega^{\sharp}),\omega^{\sharp}\rangle + \langle S(N),N\rangle|\omega|^2\right)\sum_{k=1}^{n} \langle S(e_k),e_k\rangle & \geq 2nk_{1}^2|\omega|^2.
		\end{align*}
		\indent Therefore
		\begin{multline*}
		 \sum_{k=1}^{n} |II(e_k,\omega^{\sharp})|^2 + \sum_{k=1}^{n} |II(e_k,N)|^2|\omega|^2 
		\\ -  \sum_{k=1}^{n} Rm^{\mathcal{N}}(e_k,\omega^{\sharp},e_k,\omega^{\sharp}) - Ric^{\mathcal{N}}(N,N)|\omega|^2  \\ \leq  (4k_{n+1}^2 - 2(n+1)k_1^2)|\omega|^2 \leq 0,
		\end{multline*}
		\noindent by our pinching assumption. The result follows since the inequality above is strict over the set where $|\omega|^2 \neq 0$ which is open (and non-empty since $\omega$ is assumed to be non-trivial). Thus we get that \eqref{eqomegaN} is strictly negative and the result follows  from Theorem A.
	\end{proof}
	
	\
	\begin{rmk}
		\indent In dimension $n=2$, by similar manipulations it is possible to obtain from formula (\ref{eqauxpropomega}) a slightly better result in the sense that the weaker pinching condition
		\begin{equation*}
		\frac{k_3}{k_1} < \sqrt{\frac{5}{3}}
		\end{equation*}
		\noindent still allows one to show that every closed minimal surface in such ambient space satisfies $index(M) \geq b_{1}(M)/4$. In particular, this implies that (in the setting considered in this subsection) a two-sided closed minimal surface of index one must have genus at most two, a conclusion that is conjectured to be true in any ambient three-manifold of positive Ricci curvature. We would like to thank Fernando Cod\'a Marques for pointing out this connection.
	\end{rmk}
	
	\
	
	\subsection{Pinched three-manifolds}\label{subs:threeman}
	
	\begin{thm} \label{thmpinch3}
		Let $(\mathcal{N}^{3},g)$ be a closed Riemannian three-manifold with positive scalar curvature. Assume there is an isometric embedding of $(\mathcal{N}^{3},g)$ into some Euclidean space $\mathbb{R}^{d}$ such that 
		\begin{equation*}
		R^{\mathcal{N}} > \frac{1}{2}|\vec{H}^{\mathcal{N}}|^2.
		\end{equation*}
		\indent Then, any closed oriented minimal hypersurface $M^2$ of $\mathcal{N}^{3}$ is such that
		\begin{equation*}
		index(M) \geq \frac{1}{2d} b_1(M).
		\end{equation*}
	\end{thm}
	\begin{proof}
		The result will follow immediately from Proposition \ref{propmainb} once we verify that
		\begin{equation*}
		\int_{M} \sum_{k=1}^{2} \left(|II(e_k,\omega^{\sharp})|^2 + |II(e_k,*\omega^{\sharp})|^2\right) - R^{\mathcal{N}}|\omega|^2 dM < 0
		\end{equation*}
		\noindent for all harmonic one-forms $\omega$ on an oriented minimal surface $M^2$ in $(\mathcal{N}^3,g)$. \\
		\indent The contraction of the Gauss equation for $\mathcal{N}^3$ gives
		\begin{equation*}
		R^{\mathcal{N}} = |\vec{H}|^2 - |II|^2,
		\end{equation*}
		\noindent where $\vec{H} = tr II$ denotes the mean curvature vector of $\mathcal{N}^3$ in $\mathbb{R}^d$. Since at points where $\omega$ does not vanish $\{e_{1}=\omega^{\sharp}/|\omega^{\sharp}|,e_{2}=*\omega^{\sharp}/|\omega^{\sharp}|\}$ is an orthonormal basis of the tangent space, we have
		\begin{multline*}
		\sum_{k=1}^{2} \left(|II(e_k,\omega^{\sharp})|^2 + |II(e_k,*\omega^{\sharp})|^2\right) - R^{\mathcal{N}}|\omega|^2  \\
		= \sum_{i,j=1}^{2}(|II(e_i,e_j)|^2 - R^{\mathcal{N}})|\omega|^2 \\ \leq (|II|^2 - R^{\mathcal{N}})|\omega|^2  \leq  (|\vec{H}|^2 - 2R^{\mathcal{N}})|\omega|^2 \leq 0,
		\end{multline*}
		\noindent by the pinching assumption. Once again this inequality is strict when $|\omega|^2 \neq 0$ which occurs on a non-empty open set. Thus once we integrate the result follows.
	\end{proof}		
	
	\appendix
	
	\section{The borderline case in complex projective spaces}\label{app:borderline}				
	
	Let $(\mathcal{N}^{n+1},g)$ be the complex projective space of (real) dimension $2m$ endowed with the Fubini-Study metric, namely $(\mathbb{C}\mathbb{P}^m, g_{FS})$. Let us recall here that such complex manifold is K\"ahler, which implies that the associated almost-complex structure $J$ is parallel with respect to the Levi-Civita connection (i. e. $\nabla J=0$), a fact that we are about to exploit in the sequel of this appendix. \\

	\indent If $M^n$ is a closed, embeded minimal hypersurface in such ambient manifold, we have already proven in Subsection \ref{subs:rankone} the inequality
	\begin{multline*}
	\int_{M} \left[\sum_{k=1}^{n}|II(e_k,\omega^{\sharp})|^2 + \sum_{k=1}^{n}|II(e_k,N)|^2|\omega|^{2}\right]dM \\ - \int_{M} \left[\sum_{k=1}^{n} Rm^{\mathcal{N}}(e_k,\omega^{\sharp},e_k,\omega^{\sharp}) + Ric^{\mathcal{N}}(N,N)|\omega|^2\right] dM \leq 0,
	\end{multline*}
	\noindent with equality only when the harmonic form $\omega$ on $M^{n}$ is such that $Rm^{\mathcal{N}}(N,\omega^{\sharp},N,\omega^{\sharp}) = |\omega|^2+3g(\omega^{\sharp},JN)^2 = 4|\omega|^2$ and (by virtue of the Cauchy-Schwartz inequality) this happens if and only if there exists a smooth function $f:M\to\mathbb{R}$ such that $\omega^{\sharp}=f JN$. 
	
	The key point of this appendix, which completes the proof of Theorem \ref{thm:compl} is the following assertion:
	
	\begin{prop0}\label{prop0}
		Let $M^n$ be a closed, embedded minimal hypersurface in $(\mathbb{C}\mathbb{P}^m, g_{FS})$. Suppose that a harmonic one-form $\omega$ is proportional to $(JN)^{\flat}$ at each point of $M^n$. Then such form vanishes identically on $M^n$.
	\end{prop0}	  
	
	Let us see why. Before getting to the proof, we shall present some preliminary lemmata.
	\begin{rmk}
				From now onwards, we shall set $JN=U$ in order to have a distinguished notation for this vector. Also, we let $\eta=U^{\flat}$. Furthermore, we shall denote here (coherently with the rest of the article) by $\nabla$ the Levi-Civita connection of $(\mathcal{N}^{n+1},g)$ and by $\nabla^{M}$ the induced connection on the hypersurface $M^n$. More generally, we shall adopt a superscript/subscript when referring to differential operators on $M^n$.	
			\end{rmk}

			\begin{lemm}\label{lem:skew}
				In the setting above, the following assertions are true:
				\begin{enumerate}
					
					\item[a)]{$g(\nabla^{M} f, U)=0$;}
					\item[b)]{$f\nabla_U N=-J\nabla^{M}f$.}
				\end{enumerate}		
			\end{lemm}	
			
			\begin{proof}
				We shall start by recalling that $div_{M}(\eta)=0$, which is rather standard (see e. g. pg. 286 of \cite{Kon80}). 
				Once we know this, taking the divergence of $\omega$ (and recalling the first-order characterization of harmonicity, see Subsection \ref{subs:harm}) gives
				\begin{equation*}
				0 = div_M\omega = \eta (\nabla^{M} f) = g(U,\nabla ^{M}f),
				\end{equation*}
				\noindent that is, $f$ is constant along each flow line of the flux generated by the vector field $U=JN$.\\

				For what concerns part b), taking the exterior derivative of $\omega=f\eta $ and evaluating it for $X=JN$ and an arbitrary $Y$ (a section of $TM$), we have 
				\begin{align*}
				0 = d\omega(JN,Y) & = df(JN)\eta(Y) - df(Y)\eta(JN) + fd\eta(JN,Y) \\
				& = -df(Y)+f (JN(\eta(Y))-Y(\eta(JN))-\eta([JN,Y])) \\
				& = - df(Y)+f((\nabla^{M}_{JN}\eta) (Y)-(\nabla^{M}_{Y}\eta)(JN)) \\
				& = -df(Y)+fg(\nabla^{M}_{JN}JN, Y)-fg(\nabla^{M}_{Y}JN, JN) \\
				& = -df(Y)+fg(\nabla_{JN}JN,Y)-fg(\nabla_{Y} JN,JN) \\
				& = -df(Y)+fg(J\nabla_{JN}N,Y)+fg(J\nabla_Y N, JN) \\
				& = -df(Y)-f g(\nabla_{JN}N, JY)+f g(\nabla_{Y}N, N) \\
				& = - df(Y) -f g(\nabla_{JN}N, JY)
				\end{align*}
				(since: $\eta(JN)=1$ identically and $g(\nabla_Y N,N)=0$ because $g(N,N)=1$ identically, and we have used that $J$ is a skew-symmetric operator and the ambient manifold is K\"ahler, which ensures that $\nabla J =0$). 
				\noindent That is, we have obtained $f\nabla_{JN}N = -J\nabla^{M}f$, which completes the proof. 
			\end{proof}	
			
			\begin{lemm}\label{lem:gauss}
				In the setting above, we have
				\[
				Ric^{M}(\omega^{\sharp},\omega^{\sharp})= (2m-2)f^2-|\nabla^{M}f|^2.
				\]
			\end{lemm}	
			
			\begin{proof}
				Tracing the Gauss equation \eqref{eqGaussM} gives, when $M^n$ is minimal, the equation
				\[
				Ric^{M}(U,U)= Ric^{\mathcal{N}}(U,U)-Rm^{\mathcal{N}}(U,N,U,N)- |A(U,\cdot)|^2. 
				\]
				
				Now, it has been recalled that in our setting
				\[
				Rm^{\mathcal{N}}(N,\omega^{\sharp},N,\omega^{\sharp}) = 4|\omega|^2
				\]
				so that, using also the fact that $(\mathbb{C}\mathbb{P}^m,g_{FS})$ is Einstein with constant $n+3=2m+2$,   we obtain 
				\[
				Ric^{M}(U,U)=(2m-2)- |A(U,\cdot)|^2 
				\] 
				which proves the asserted identity once we make use of part b) of Lemma \ref{lem:skew}.
			\end{proof}	
			
			\begin{lemm}\label{lem:poinc}
					In the setting above,
					\[
					|\nabla^{M}f|^2 \leq |\nabla^{M}\omega|^2.
					\]
					\end{lemm}		
					
					\begin{proof}
						Differentiating $\omega^{\sharp} = f(JN)$, we have $\nabla^{M}_{X}\omega^{\sharp} = (\nabla^{M}_{X}f)JN + f\nabla^{M}_{X}JN$ for every vector field $X$ tangent to $M^n$. Therefore
						\[
						|\nabla^{M}_{X}\omega^{\sharp}|^2 =  |\nabla^{M}_{X}f|^2 +2f(\nabla^{M}_{X}f)g(JN,\nabla_{X}^{M}JN) + f^2|\nabla^{M}_{X}JN|^2.
						\]
						
						Since $2g(JN,\nabla_{X}^{M}JN) = Xg(JN,JN)= Xg(N,N)=0$, we obtain
						\[
						|\nabla^{M}_{X}\omega^{\sharp}|^2 =  |\nabla^{M}_{X}f|^2 + f^2|\nabla^{M}_{X}JN|^2 \geq |\nabla^{M}_{X}f|^2
						\]
						for all tangent vector fields $X$. The conclusion follows.	
					\end{proof}

			Now, let us proceed with the proof of Proposition A.0. 
			
			\begin{proof}
				The standard Bochner identity for one-forms asserts that
				\[
				\Delta_1 \alpha = - \Delta \alpha +Ric^{M}(\alpha^{\sharp},\cdot)
				\]	
				for any smooth one-form $\alpha$. To avoid ambiguities, let us observe that $\Delta=\Delta^{M}$ stands for the standard/elementary Laplace-Beltrami operator on $M$. 
				Hence, if $\omega$ is harmonic, as we are assuming, then
				\[
				\Delta \omega = Ric^{M} (\omega^{\sharp},\cdot).
				\]
				
				This is an equation between one-forms so we can \textsl{choose} the vector field to test it with, so let us evaluate both sides on $\omega^{\sharp}$. If we do so and then integrate over the manifold $M$ we get to the identity
				\[
				(2m-2)\int_{M}| \omega|^2\,dM+ \int_{M} (|\nabla^{M}\omega|^2-|\nabla^{M} f|^2)\,dM=0
				\]where we have used Lemma \ref{lem:gauss} for the right-hand side.
				Since $m\geq 2$, then Lemma \ref{lem:poinc} ensures that $\omega$ is the trivial one-form and we are done. 
			\end{proof}

\end{document}